\documentclass[12pt]{amsart}
\usepackage[T1]{fontenc}
\usepackage[utf8]{inputenc}
\usepackage{amsmath,amssymb}
\usepackage{mathrsfs}
\usepackage{amscd}
\usepackage[all,cmtip]{xy}
\usepackage{bbm}
\usepackage{bbding}
\usepackage{newtxtext}
\usepackage{newtxmath}
\usepackage[shortlabels]{enumitem}
\usepackage{tikz}
\usetikzlibrary{calc}
\usetikzlibrary{arrows,shapes,chains}
\usepackage{tikz-cd}
\usepackage[margin=1.15in]{geometry}
\usepackage[colorlinks,final,backref=page,hyperindex]{hyperref}

\newtheorem{thmA}{Theorem}

\newtheorem{thm}{Theorem}[section]
\newtheorem{prop}[thm]{Proposition}

\newtheorem{prop-def}{Proposition-Definition}[section]

\theoremstyle{definition}
\newtheorem{defn}[thm]{Definition}

\newtheorem{remark}[thm]{Remark}

\theoremstyle{plain}
\newtheorem{theorem}[thm]{Theorem}

\newtheorem{proposition}[thm]{Proposition}

\newcommand{\nc}{\newcommand}
\nc{\calN}{\mathcal{N}}
\nc{\Hom}{\mathrm{Hom}}
\nc{\id}{\mathrm{id}}
\nc{\ot}{\otimes}
\nc{\rmC}{\mathrm{C}}
\nc{\rmChar}{\mathrm{Char}}
\nc{\rmHH}{\mathrm{HH}}

\title{A Symmetric Counterexample to the Snashall--Solberg Conjecture}
\author{Kai Wang and Guodong Zhou}
\date{\today}

\begin{document}

	\maketitle

	\allowdisplaybreaks

	\begin{abstract}
		It is shown that the trivial extension of the Xu--Snashall algebra is a symmetric counterexample to the Snashall--Solberg conjecture, which states that the Hochschild cohomology ring of a finite-dimensional algebra modulo nilpotence is a finitely generated algebra. To the best of our knowledge, this is the first known selfinjective counterexample.
	\end{abstract}

	\section{Introduction}

	Motivated by their work on support varieties via Hochschild cohomology \cite{SS04, EHTSS04},
	Snashall and Solberg \cite{SS04} conjectured that the Hochschild cohomology ring modulo nilpotence, that is, the quotient of the Hochschild cohomology ring of a finite-dimensional algebra by the ideal generated by its homogeneous nilpotent elements, is a finitely generated algebra. The Snashall--Solberg conjecture has been shown to hold for many classes of algebras \cite{GS06,GSS03,GSS06, SS11}.

	However, Xu \cite{Xu08} constructed a counterexample to the Snashall--Solberg conjecture over a base field of characteristic two, and Snashall \cite{Sna09} generalized this example to arbitrary characteristic. Further counterexamples to the Snashall--Solberg conjecture have been obtained in \cite{HF12, XZ12}. Nevertheless, neither the Xu--Snashall algebra nor any of these subsequent examples is selfinjective. Consequently, the Snashall--Solberg conjecture has remained open for selfinjective algebras.

	Let us recall the definition of the Xu--Snashall algebra in terms of a quiver with relations. Let $K$ be a field. Let $Q$ be the quiver
	\[
	\begin{tikzcd}
		1 \arrow[loop above, out=120, in=60, looseness=4,"a"] \arrow[loop below, out=300, in=240, looseness=5,"b"] \arrow[r,"c"] & 2
	\end{tikzcd}
	\]
	and let $I$ be the ideal of the path algebra $KQ$ generated by the relations $\{a^2, b^2, ab-ba, ac\}$. The Xu--Snashall algebra is defined to be $A=KQ/I$.

	\begin{thmA}[\cite{Xu08,Sna09}]
		\label{thm:xu-snashall}
		There is an isomorphism of graded \(K\)-algebras
		\[
		\rmHH^*(A)/\calN
		\cong
		\begin{cases}
			K\oplus K[x,y]y,
			&\operatorname{char}K=2,\\[2mm]
			K\oplus K[x^2,y^2]y^2,
			&\operatorname{char}K\neq 2.
		\end{cases}
		\]
		Here $\calN$ denotes the ideal generated by homogeneous nilpotent elements; \(y\) has cohomological degree $1$ and \(xy\) has cohomological degree $2$.
	\end{thmA}

	The algebras appearing in Theorem~\ref{thm:xu-snashall} are well-known examples of non-Noetherian algebras.

	Our main result reads as follows.

	\begin{thmA}
		\label{Theorem: selfinjective counterexample}
		Let \(K\) be a field of arbitrary characteristic, and let \(T(A)\) be the   trivial extension of the Xu--Snashall algebra $A$. Then the quotient
		$
		\rmHH^*(T(A))/\calN$
		is not finitely generated as a \(K\)-algebra.
	\end{thmA}

	In particular, \(T(A)\) provides a symmetric counterexample to the Snashall--Solberg conjecture. To the best of our knowledge, this is the first known counterexample to the Snashall--Solberg conjecture which is a selfinjective algebra.

	To conclude the introduction, let us discuss some recent developments. Let $G$ be the weak Gerstenhaber ideal generated by all homogeneous nilpotent elements of $\rmHH^*(\Lambda)$; recall that a weak Gerstenhaber ideal is an ideal for the cup product which is at the same time a Lie subalgebra for the Gerstenhaber bracket. Hermann \cite{Her16} asked further whether the quotient $\rmHH^*(\Lambda)/G$ is a finitely generated algebra, and suggested considering the Xu--Snashall algebra first.

	Recall that Snashall \cite{Sna09} computed $\rmHH^*(A)/\calN$ using the fact that it is isomorphic to the graded center of the Koszul dual algebra modulo nilpotence. Oke \cite{Oke22} also considered the Xu--Snashall algebra and showed that $G=\calN$ by means of the homotopy lifting method of Volkov \cite{Vol19}, thereby answering Hermann's question in the negative; however, he did not determine the Gerstenhaber algebra structure on $\rmHH^*(A)$. In a recent preprint, Long, Shi and Zhou \cite{LSZ26} computed the Gerstenhaber algebra structure on $\rmHH^*(A)$ explicitly in terms of generators and relations, from which the equality $G=\calN$ follows directly.

	It would be very interesting to determine explicitly the Gerstenhaber algebra structure on the Hochschild cohomology ring of the trivial extension of the Xu--Snashall algebra; however, this task seems to be rather difficult. We have tried to achieve this with the help of the AI assistants ChatGPT and Kimi, but without success.

	\section{Proof of Theorem~\ref{Theorem: selfinjective counterexample}}

	\begin{defn}[\cite{Hoc45}]
		Let \(\Lambda\) be a $K$-algebra. The Hochschild cochain complex of $\Lambda$ is
		\[
		\rmC^\bullet(\Lambda):=\bigoplus_{n=0}^\infty \rmC^n(\Lambda),
		\]
		where $\rmC^n(\Lambda)=\Hom_K(\Lambda^{\ot n},\Lambda)$ and the differential $\delta^n_\Lambda: \rmC^n(\Lambda)\rightarrow \rmC^{n+1}(\Lambda)$ is defined by
		\[
		\begin{aligned}
			\delta^n_\Lambda(f)(a_1\otimes \cdots \otimes a_{n+1})
			={}& (-1)^{n+1} a_1f(a_2\otimes \cdots \otimes a_{n+1})\\
			&+\sum_{i=1}^n(-1)^{n+1-i}f(a_1\otimes\cdots\otimes a_{i-1}\ot a_i\cdot a_{i+1}\ot a_{i+2}\otimes\cdots\ot a_{n+1})\\
			&+ f(a_1\otimes \cdots \ot a_n)a_{n+1}
		\end{aligned}
		\]
		for all $f\in \rmC^n(\Lambda)$ and $a_1,\dots,a_{n+1}\in \Lambda$.
		The cohomology of the Hochschild cochain complex $\rmC^\bullet(\Lambda)$ is called the Hochschild cohomology of $\Lambda$, denoted by $\rmHH^*(\Lambda)$.
	\end{defn}

	\begin{defn}
		For \(f\in \rmC^m(\Lambda)\) and \(g\in \rmC^n(\Lambda)\), their cup product $f\cup g\in \rmC^{m+n}(\Lambda)$ is defined by
		\[
		(f\cup g)(a_1\otimes \ldots\otimes a_{m+n})
		=(-1)^{mn}
		f(a_1\otimes \ldots\ot a_m)g(a_{m+1}\ot \ldots\ot a_{m+n}).
		\]
	\end{defn}

	\begin{prop}[\cite{Ger63}]
		The Hochschild cochain complex $(\rmC^\bullet(\Lambda),\delta^\bullet_\Lambda)$, endowed with the cup product ``$\cup$'', forms a differential graded algebra. In particular, $(\rmHH^*(\Lambda),\cup)$ is a graded algebra.
	\end{prop}

	We now establish a comparison result for the Hochschild cohomology of two algebras connected by a suitable pair of algebra homomorphisms.

	Let \(\Lambda\) and \(\Gamma\) be \(K\)-algebras, and denote their Hochschild cochain complexes by \((\rmC^\bullet(\Lambda),\delta_\Lambda^\bullet)\) and \((\rmC^\bullet(\Gamma),\delta_\Gamma^\bullet)\), respectively. Suppose that there exist algebra homomorphisms
	\[
	p:\Gamma\to \Lambda,\qquad \iota:\Lambda\to \Gamma
	\]
	such that \(p\circ\iota=\id_\Lambda\). For each \(n\ge 1\), define
	\[
	\Phi^n : \rmC^n(\Gamma) \to \rmC^n(\Lambda), \qquad
	\Phi^n(f) = p\circ f\circ \iota^{\otimes n},
	\]
	and set \(\Phi^0 = p\).

	\begin{proposition}[{\cite[Theorem 2.3]{AGST16}}]\label{prop:cochain-projection}
		With the notation above, 
			\[
			\Phi^\bullet : (\rmC^\bullet(\Gamma),\delta_\Gamma^\bullet,\cup) \longrightarrow (\rmC^\bullet(\Lambda),\delta_\Lambda^\bullet,\cup)
			\]
			is a morphism of differential graded algebras.
  Consequently, taking cohomology yields a morphism of graded algebras
			\[
			\varphi^*=H^*(\Phi^\bullet):
			(\rmHH^*(\Gamma),\cup) \longrightarrow (\rmHH^*(\Lambda),\cup).
			\]
		 
	\end{proposition}

	\begin{remark}\label{rem:gerstenhaber}
		Proposition~\ref{prop:cochain-projection} concerns only the Hochschild differential and the cup product. It gives no compatibility with the cochain insertions that define the Gerstenhaber bracket. Indeed, the Hochschild projection need not preserve the Gerstenhaber graded Lie bracket \cite[Example~2.5]{AGST16}.
	\end{remark}

	Next, we recall the construction of the trivial extension of an algebra; for the Hochschild cohomology of trivial extensions, we refer the reader to \cite{CibilsMarcosRedondoSolotar2003}.

	Let \(\Lambda\) now be a finite-dimensional $K$-algebra, and let
	\[
	D(\Lambda)=\operatorname{Hom}_K(\Lambda,K)
	\]
	be its $K$-dual, endowed with the standard \(\Lambda\)-\(\Lambda\)-bimodule structure
	\[
	(a\cdot f\cdot b)(x)=f(bxa)
	\qquad \text{for } a,b,x\in \Lambda,\ f\in D(\Lambda).
	\]
	The \emph{trivial extension} of \(\Lambda\), denoted by \(T(\Lambda)\), is the associative \(K\)-algebra with underlying vector space \(\Lambda \oplus D(\Lambda)\) and with multiplication given by
	\[
	(x,f)(y,g)=\bigl(xy,\; x\cdot g+f\cdot y\bigr)
	\]
	for all \(x,y\in \Lambda\) and \(f,g\in D(\Lambda)\).
	It is well known that the trivial extension $T(\Lambda)$ is a symmetric algebra via the nondegenerate symmetric associative bilinear form
	\[
	\langle (a, f), (b, g)\rangle=f(b)+g(a).
	\]
	Moreover, the natural projection $p: T(\Lambda)\to \Lambda$, $(x,f)\mapsto x$, and the natural injection $\iota:\Lambda\to T(\Lambda)$, $x\mapsto (x,0)$, are algebra homomorphisms such that \(p\circ\iota=\id_\Lambda\).

	We shall also need the following comparison result relating the Hochschild cohomology of an associative algebra to that of its trivial extension.

	\begin{theorem}[{\cite[Theorem A]{AGST16}}] \label{Theorem: surjection of Hochschild cohomology}
		Let \(\Lambda\) be a finite-dimensional $K$-algebra. The map introduced in Proposition~\ref{prop:cochain-projection}
		\[
		\varphi^*:\rmHH^*(T(\Lambda))\longrightarrow\rmHH^*(\Lambda)
		\]
		is a surjective homomorphism of graded algebras.
	\end{theorem}

	We are now in a position to prove the main result.

	\begin{proof}[Proof of Theorem~\ref{Theorem: selfinjective counterexample}]
		By Theorem~\ref{Theorem: surjection of Hochschild cohomology}, the maps $p$ and $\iota$ induce a surjective homomorphism of graded algebras
		\[
		\varphi^*:\rmHH^*(T(A))\longrightarrow\rmHH^*(A).
		\]
		Being a homogeneous algebra homomorphism, $\varphi^*$ sends homogeneous nilpotent elements to nilpotent elements; hence it maps the ideal $\calN$ of $\rmHH^*(T(A))$ into the ideal $\calN$ of $\rmHH^*(A)$, and therefore induces a surjective homomorphism of graded algebras
		\[
		\overline{\varphi^*}:\rmHH^*(T(A))/\calN\longrightarrow\rmHH^*(A)/\calN.
		\]
		If $\rmHH^*(T(A))/\calN$ were finitely generated as a $K$-algebra, then so would be its quotient $\rmHH^*(A)/\calN$, contradicting Theorem~\ref{thm:xu-snashall}. Consequently, $\rmHH^*(T(A))/\calN$ is not finitely generated as a $K$-algebra.
	\end{proof}

	\bigskip

	\noindent\textbf{Acknowledgements.}
	This work was supported by the National Key R\&D Program of China (No.~2024YFA1013803) and by the Shanghai Key Laboratory of PMMP (No.~22DZ2229014).

	The precise form of the counterexample was conceived by the authors, while the method of proof was suggested by the AI assistant Kimi. We express our sincere gratitude to Kimi for enlarging mankind's knowledge.

\end{document}